\documentclass[oneside,a4paper,reqno,12pt]{amsart}
\usepackage{amsfonts}
\usepackage{amssymb}
\usepackage{amsmath}
\usepackage{amsfonts}
\usepackage{amssymb}
\usepackage{graphicx,color}

\newtheorem{theorem}{Theorem}[section]

\newtheorem{proposition}[theorem]{Proposition}
\newtheorem{corollary}[theorem]{Corollary} 
\theoremstyle{remark}
 
\theoremstyle{definition}
 
\newtheorem{example}[theorem]{Example} 

\def\r{\mathbb R}

\begin{document}
\title{Translators of the Gauss curvature flow}
\author{Muhittin Evren Aydin}
\address{Department of Mathematics, Faculty of Science. Firat University. Elazig 23200 Turkey}
\email{meaydin@firat.edu.tr}
\author{Rafael L\'opez}
\address{Departamento de Geometr\'{\i}a y Topolog\'{\i}a. Universidad de Granada. Avenida Fuentenueva, s/n. Granada 18071 Spain}
\email{rcamino@ugr.es}

\begin{abstract}
A $K^\alpha$-translator is a surface in Euclidean space $\r^3$ that moves by translations in a spatial direction and under the $K^\alpha$-flow, where $K$ is the Gauss curvature and $\alpha$ is a constant. We classify all $K^\alpha$-translators that are rotationally symmetric. In particular, we prove that for each $\alpha$ there is  a $K^\alpha$-translator intersecting orthogonally the rotation axis. We also describe all $K^\alpha$-translators invariant by a uniparametric group of helicoidal motions and the translators obtained by separation of variables.
\end{abstract}
\maketitle


\section{Introduction and results}
 
 The flow by powers of the Gauss curvature $K$ was initiated by Chow \cite{ch} after the articles of Firey and Tso (\cite{fi,ts}). These works were the starting point of the theory of the flow by the Gaussian curvature (\cite{an,ur}), a topic of high activity in geometric analysis that continues to the present. Given a smooth  immersion $X:\Sigma\to\r^3$ of a strictly convex surface $\Sigma$ in Euclidean space $\r^3$, we consider the $K^\alpha$-flow as a one-parameter family of smooth immersions $X_t=X(\cdot,t)\colon\Sigma\to \r^3$, $t\in [0,T)$ such that $X_0=X$ and satisfying the flow 
$$\frac{\partial}{\partial t}X(p,t)=-K(p,t)^\alpha N(p,t),\quad (p,t)\in \Sigma\times [0,T),$$
where $\alpha\in\r$ is a constant, $N(p,t)$ is the unit normal of $X(p,t)$ and $K(p,t)$ is the Gauss curvature at $X(p,t)$. An interesting problem is the evolution of a surface through the flow. As an example for the reader, if $\alpha=1$ the surface becomes spherical (\cite{an3}). 

If the surface moves under the $K^\alpha$-flow along a spatial direction $\vec{v}\in \r^3$, then $X_0$ satisfies $K^\alpha=\lambda\langle N,\vec{v}\rangle$ for some positive constant $\lambda$. The vector $\vec{v}$ is assumed to be unitary and it is called the {\it speed} of the flow. After a dilation,  we can assume that  $\lambda=1$. A surface $\Sigma$ is called  a {\it translator} by the $K^\alpha$-flow with speed $\vec{v}$, or simply, a {\it $K^\alpha$-translator} in case that $\vec{v}$ is understood,  if
\begin{equation}\label{k1}
K^\alpha= \langle N,\vec{v}\rangle.
\end{equation}
The notion of translator   by positive powers of the Gauss curvature  appeared in \cite{ur2}. See also \cite{cdl,le}.

In this paper we want to find examples of $K^\alpha$-translators   under the geometric condition that the surface is defined kinematically as the movement of a curve by a uniparametric family of rigid motions of $\r^3$. Following Darboux   \cite[Livre I]{da}, we consider surfaces parametrized by   $\Psi(s,t)=A(t)\cdot \gamma(s)+\beta(t)$, where $\gamma$ and $\beta$ are two spatial  curves and $A(t)$ is an orthogonal matrix. The cases that we are interesting are: rotational surfaces ($A(t)$ is a uniparametric group of rotations and $\beta$ is constant), helicoidal surfaces ($A(t)$ is a uniparametric group of helicoidal motions), translation surfaces ($A(t)$ is the identity) and  ruled surfaces ($\gamma$ is a straight-line). Since we have left to one side the study of the evolution of the surface,  we do not require that $K$ to be positive but only that $K^\alpha$ has sense. For example, $K$ may be negative if $\alpha\in\mathbb{Z}$. This will be implicitly assumed in all the  results. 

In the flow by powers of the Gaussian curvature, it is known that the $K^{1/4}$-flow is special because it has an interpretation in affine differential geometry as the affine normal flow of a convex surface (\cite{an2,ca}). However, in base of our proofs, this case will appear as natural because when we express the equation \eqref{k1} for each one of the above three types of surfaces and  when $\alpha$ is exactly $1/4$, this   equation is greatly simplified due to a cancellation of terms. This makes the arguments easier than in the general case of $\alpha$.

Planes are examples of $K^\alpha$-translators for any $\alpha>0$ provided the plane is parallel to the speed of the flow. Throughout this paper, we will discard planes as examples of $K^\alpha$-translators. Other examples that we will not consider in the class of $K^\alpha$-translators are the surfaces whose Gauss curvature is constant, where now the equation \eqref{k1} says that the unit normal vector $N$ makes a constant angle with a fixed direction. Similarly, we discard the case $\alpha=0$.

It is important to point out that there is not a priori relation between the speed vector $\vec{v}$ and the special parametrization of each of the above Darboux surfaces. So, if one considers the study of translators one can prescribed the speed $\vec{v}$, usually, the vector $(0,0,1)$ in the literature. However, if one considers one of the above types of Darboux surfaces, the parametrization has no relation with $\vec{v}$. This can be  clearly seen for rotational surfaces (also helicoidal surfaces). The rotation axis of the surface is, initially, independent of the vector $\vec{v}$. However, as we will see, for rotational and helicoidal surfaces, the  speed must be parallel to the axis (Propositions  \ref{p-arot} and \ref{pr-ahel}). Something similar occurs for translations surfaces where, if one prescribes that the surface is $z=f(x)+g(y)$, the speed may be arbitrary.

The organization of this paper is as follows. In Section \ref{sec2}, we obtain all rotational $K^\alpha$-translators. Recall that in \cite{ur2}, Urbas obtains these surfaces for $\alpha\in (0,1/2]$ using the Legendre transform (see also \cite{ju} for a similar calculation). Our approach uses simple geometric arguments and it holds for any $\alpha$ (Theorem \ref{t-1}). In particular, for each value of $\alpha$, we prove the existence of rotational examples intersecting orthogonally the rotation axis (Corollary \ref{co-1}). Section \ref{sec22} is devoted to the study of $K^\alpha$-translators of helicoidal type. Although we are not able to obtain explicit parametrizations of these surfaces, we do a first integration of the generating curve (Theorem \ref{t-hel}), also in terms of the Bour function (Theorem \ref{t-hel2}). In Section \ref{sec3}, we obtain the classification of all solutions of \eqref{k1} obtained by separation of variables $z=f(x)+g(y)$, where $(x,y,z)$ are the canonical coordinate system of $\r^3$. Although these solutions depend on a particular choice of coordinates of $\r^3$, our result holds for any speed vector $\vec{v}$. In particular, we prove that there are $K^\alpha$-translators only if $\alpha=1/4$ (Theorems \ref{t-31} and \ref{t-32}). Besides we provide new examples of $K^{1/4}$-translators of type $z=f(x)+g(y)$, we also give new examples of $K^{1/4}$-translators obtained by separation of variables of type $z=f(x)g(y)$ (Example \ref{ex-homo}). Finally, in Section \ref{sec4}, we investigate the existence of ruled $K^\alpha$-translators, proving that there are not ruled $K^\alpha$-translators, except trivial cases (Theorem \ref{t-4}).

The authors have extended the results of this paper for (spacelike and timelike) $K^\alpha$-translators in Lorentz-Minkowski space (\cite{alo}).

\section{Rotational $K^\alpha$-translators}\label{sec2}

In this section, we classify  the surfaces of
revolution satisfying \eqref{k1}. A first question is if there is a relation
between the rotation axis of the surface and the speed vector $\vec{v}$. As it is expectable, we prove that  the vector $\vec{v}$ must be parallel to the rotation axis.

\begin{proposition}\label{p-arot}
Let $\Sigma$ be a $K^\alpha$-translator. If $\Sigma$ is a surface
of revolution, then  the rotation axis is parallel to the speed vector $\vec{v}$.
\end{proposition}
\begin{proof}
After a rigid motion, we can assume that  the
rotation axis $L$ is the $z$-axis. The generating curve of a rotational surface is a  curve included in the coordinate $xz$-plane which can be assumed that it is a graph on the $x$-axis, except in the case that this curve is a straight-line parallel to the $z$-axis. In this particular situation, the surface is a circular cylinder where we know that $K=0$. Thus a circular cylinder satisfies \eqref{k1} if $\langle N,\vec{v}\rangle=0$, that is, the vector $v$ is parallel to the $z$-axis, proving the result for this particular case.  

Suppose now the general case that the generating curve writes as $z=f(x)$ where $f\colon I\subset\r^+\to\r$ is a smooth function. Then a parametrization of $\Sigma$ is  
\begin{equation}\label{p-rot}
X(r,\theta )=(r\cos \theta ,r\sin \theta ,f(r)),
\end{equation}
where  $\theta \in \mathbb{R}$.   The unit
normal vector $N$ is 
\begin{equation*}
N(r,\theta)= \frac{1}{\sqrt{1+f'^2}}(-f'\cos \theta
,-f'\sin \theta ,1)
\end{equation*}
 and the principal curvatures of $\Sigma$ are 
\begin{equation*}
\kappa_1=\frac{f''}{(1+f'^2)^{3/2}},\quad \kappa_2=%
\frac{f'}{r(1+f'^2)^{1/2}}.
\end{equation*}
If $\vec{v}=(v_1,v_2,v_3)$, then \eqref{k1} writes as 
\begin{equation*}
K^\alpha=(\kappa_1\kappa_2)^\alpha=\left(\frac{f'f''}{r(1+f'^2)^2}%
\right)^\alpha=\frac{1}{\sqrt{1+f'^2}}(-v_1f'\cos \theta -v_2f'\sin \theta +v_3),
\end{equation*}%
or equivalently%
\begin{equation*}
A_0 +A_1 \cos \theta +A_2 \sin \theta =0,
\end{equation*}%
where 
\begin{equation*}
\begin{split}
A_0 &=\left(\frac{f'f''}{r(1+f'^2)^2}%
\right)^\alpha -v_3(1+f'^2)^{-1/2},\\
 A_1 &=v_1f'(1+f'^2)^{-1/2},\quad A_2 =v_2f'(1+f'^2)^{-1/2}.
 \end{split}
\end{equation*}%
Since the functions $\{ 1,\cos \theta ,\sin \theta \} $ are linearly independent,
we  conclude $A_0 =A_1 =A_2 =0$.  If $f'$ is constantly $0$, then $f$ is constant and the surface is a plane. This case was initially discarded. Therefore  from $A_1 =A_2 =0$, we find that $v_1=v_2=0$, concluding that $\vec{v}$ is parallel to 
$L$.
\end{proof}

After a rigid motion, we will assume that the rotation axis is the $z$-axis and consequently from this proposition, that $\vec{v}=(0,0,1)$ after a symmetry about the $xy$-plane if necessary. Then %
\eqref{k1} is 
\begin{equation*}
\left( \frac{f'f''}{r(1+f'^2)^2}\right)
^{\alpha }=\frac{1}{(1+f'^2)^{1/2}}.
\end{equation*}%
Set $g=f'/(1+f'^2)^{1/2}$. Notice that $\kappa_1=g'$ and $r\kappa_2=g$. In terms of $g$, the above equation becomes    
\begin{equation*}
\frac{g}{(1-g^{2})^{\frac{1}{2\alpha }}}g'=r.
\end{equation*}%
A first  integration gives 
\begin{equation}\label{domain}
g^2=\left\{ 
\begin{array}{lll}
1-me^{-r^2},m>0, &  &\alpha =\frac{1}{2} \\ 
1-\left( m-\frac{2\alpha-1 }{2\alpha }r^{2}\right) ^{\frac{2\alpha }{2\alpha
-1}},m\in \mathbb{R}, & &\alpha \neq \frac{1}{2}.%
\end{array}%
\right.
\end{equation}
In the case $\alpha=1/2$, the condition $g*2\geq 0$ says that the domain of $r$ is when $r^2\geq \log(m)$. Otherwise, we need to distinguish if $2\alpha/(2\alpha-1)$ is negative or positive, that is, $\alpha$ belongs to $(0,1/2)$ or not. If $\alpha\in (0,1/2)$, the parenthesis in \eqref{domain} must be positive, yielding $r^2>2\alpha/(2\alpha-1)m$. On the other hand, using now that $g^2\geq 0$, we have $r^2\geq 2\alpha/(2\alpha-1)(m-1)\geq 0$, so this is the restriction on $r$ because $2\alpha/(2\alpha-1)<0$. If $\alpha\not\int [0,1/2]$, then $2\alpha/(2\alpha-1)>0$. Since the parenthesis in \eqref{domain} must be positive because $g^2<1$, then we obtain $r^2<2\alpha/(2\alpha-1)m$ and from the fact that $g^2\geq 0$, the restriction is $r^2\geq 2\alpha/(2\alpha-1)(m-1)$.  To summarize, we have 
\begin{equation}\label{c-rot}
\left\{\begin{array}{lll}
\frac{2\alpha}{2\alpha-1}(m-1)\leq r^2,& &\alpha\in (0,\frac12)\\
\frac{2\alpha}{2\alpha-1}(m-1)\leq r^2<\frac{2\alpha}{2\alpha-1}m,&& \alpha\not\in [0,\frac12].
\end{array}\right.
\end{equation}
Hence we deduce
\begin{equation}\label{d-rot}
f'=\left\{ 
\begin{array}{ll}
\pm\left(\frac{1}{m}e^{r^2}-1\right)^{1/2},  & \alpha =\frac{1}{2} \\ 
\pm\left(\left(m-\frac{2 \alpha -1}{2 \alpha }r^2\right)^{\frac{2 \alpha }{1-2 \alpha }}-1\right)^{1/2},& \alpha \neq \frac{1}{2}.%
\end{array}%
\right.
\end{equation}%
As conclusion, we have the classification of all $K^\alpha$-translators  that also are surfaces of revolution.

\begin{theorem}\label{t-1} Let $\Sigma$ be a $K^\alpha$-translator. If $\Sigma$ is a surface of revolution about the $z$-axis, then $\Sigma$ is a circular cylinder  of arbitrary radius or $\Sigma$ parametrizes as \eqref{p-rot} where  
\begin{equation}\label{p-rot2}
f(r) =\left\{ 
\begin{array}{lll}
\pm \int^{r}\left( \frac{1}{m}e^{t^{2}}-1\right) ^{1/2}\, dt, m>0,& &\alpha =\frac{1%
}{2} \\ 
\pm \int^{r}\left( \left( m-\frac{2\alpha-1 }{2\alpha }t^2\right) ^{\frac{%
2\alpha }{1-2\alpha }}-1\right) ^{1/2}\, dt,m\in \mathbb{R}, & &\alpha \neq \frac{%
1}{2}.%
\end{array}%
\right.
\end{equation}
Furthermore,   the maximal  domain of the  function $f(r)$ is  
\begin{enumerate}
\item $[\sqrt{\log{m}},\infty)$, if   $\alpha=1/2$.
\item $[\sqrt{\frac{2\alpha}{2\alpha-1}m},\infty)$, if   $\alpha\in (0,1/2)$.
\item  $[\sqrt{\frac{2\alpha}{2\alpha-1}(m-1)},\sqrt{\frac{2\alpha}{2\alpha-1}m})$, if   $\alpha\not\in [0,1/2]$. In this case,  we have 
$$\lim_{r\to \sqrt{\frac{2\alpha}{2\alpha-1}(m-1)}}f'(r)=0,\quad \lim_{r\to \sqrt{\frac{2\alpha}{2\alpha-1}m}}f(r)=\infty.$$
\end{enumerate}
In all these cases, we understand that if in the radicand in the left-end of the interval is negative, then the value of  this end is $0$.
\end{theorem}

\begin{proof} It remains to prove the limits in the case (3). The first limit is consequence of \eqref{c-rot} and \eqref{d-rot}. For the second limit, note that the maximal domain of $f$ is a bounded interval and that $\lim_{r\to \sqrt{\frac{2\alpha}{2\alpha-1}m}}f(r)=\infty$ by \eqref{d-rot}.
\end{proof}

We point that it is expectable that in the case $\alpha\not\in [0,1/2]$ the domain cannot be $[0,\infty)$ because there are no entire graphs that are $K^\alpha$-translators if $\alpha>1/2$ (\cite[Sect. 4]{ur2}) and if $\alpha<0$ (\cite[Th. 6.1]{ur3}).   It deserves to note the case $\alpha=1/4$ because it is possible to integrate explicitly \eqref{p-rot2} (see also \cite{le}). See Figure \ref{fig14}.

\begin{corollary} Rotational  $K^{1/4}$-translators   form a uniparametric family of surfaces  parametrized by \eqref{p-rot}, where  
$$f(r)=\frac{1}{2} \left(r \sqrt{m+r^2-1}+(m-1) \log \left(\sqrt{m+r^2-1}+r\right)\right)+c,\quad m,c\in\r.$$
The maximal domain of $f$ is $[\sqrt{1-m},\infty)$ if $m<1$ and $[0,\infty)$ if $m\geq 1$.   For the value $m=1$, $f$ is the parabola $f(r)=r^2/2$, the graphic of $f(r)$ intersects orthogonally the rotation axis and the surface is a paraboloid.
\end{corollary}

 \begin{figure}[hbtp]
\begin{center}
\includegraphics[width=.7\textwidth]{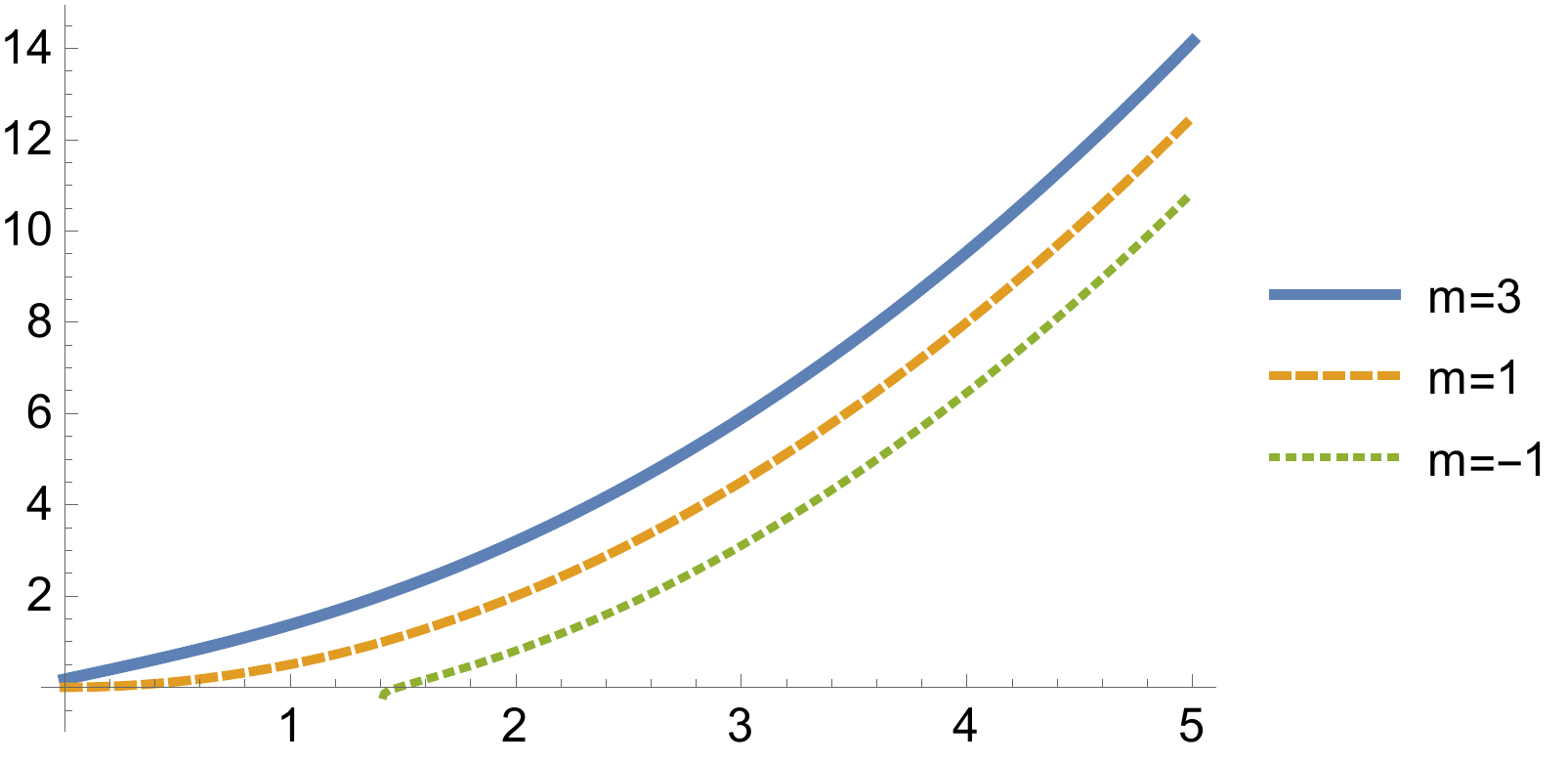}
\end{center}
\caption{Generating curves of rotational $K^{1/4}$-translators for different values of the parameter $m$.}\label{fig14}
\end{figure}

By \eqref{c-rot}, we point out that  the maximal domain of $f$ is not $[0,\infty)$ in general. However, an interesting case to investigate is if there are generating curves that meet orthogonally the rotation axis. We prove that this occurs for all cases of $\alpha$. See Figure \ref{fig2}.

\begin{corollary}\label{co-1}
For each $\alpha$, there are rotational $K^\alpha$-translators whose generating curves intersect orthogonally the rotation axis.  These surfaces are unique up to vertical translations. Furthermore, 
\begin{enumerate}
\item If $\alpha\in (0,\frac12]$, the maximal domain of $f$ is $[0,\infty)$, $\lim_{r\to\infty}f(r)=\infty$ and 
$$f(r)=(1-2\alpha)\left(\frac{1-2\alpha}{2\alpha}\right)^{\frac{\alpha}{1-2\alpha}}r^{\frac{1}{1-2\alpha}}+o(r^{\frac{1}{1-2\alpha}}).$$
\item If $\alpha\not\in [0,\frac12]$, the maximal domain is $[0,\sqrt{\frac{2\alpha}{2\alpha-1}})$, with 
$$\lim_{r\to\sqrt{\frac{2\alpha}{2\alpha-1}}}f(r)=\infty.$$
\end{enumerate}
\end{corollary}

\begin{proof} 
The condition on the orthogonality with the rotation axis  requires  that $f$ is defined at $r=0$ and $f'(0)=0$. From \eqref{d-rot}, it is immediate that $m$ must be $1$ and the same occurs in the particular case  $\alpha=1/2$. This solution is $C^2$ at $r=0$ because from \eqref{d-rot}, we have $\lim_{r\to 0}f''(r)=1$. Hence  $f$ is $C^\infty$ in its domain by regularity (\cite{caf,gt}). The uniqueness is consequence of the solvability of \eqref{p-rot2}.

The behaviour of $f$ at infinity is consequence of \eqref{d-rot} and the L'H\^{o}pital rule. Indeed, if $\delta=1/(1-2\alpha)$, then 
$$\lim_{r\to\infty}\frac{f(r)}{r^\delta}=\lim_{r\to\infty}\frac{f'(r)}{\delta r^{\delta-1}}=\frac{1}{\delta} \left(\frac{1-2\alpha}{2\alpha}\right)^{\frac{\alpha}{1-2\alpha}}.$$
\end{proof}

 \begin{figure}[hbtp]
\begin{center}
\includegraphics[width=.34\textwidth]{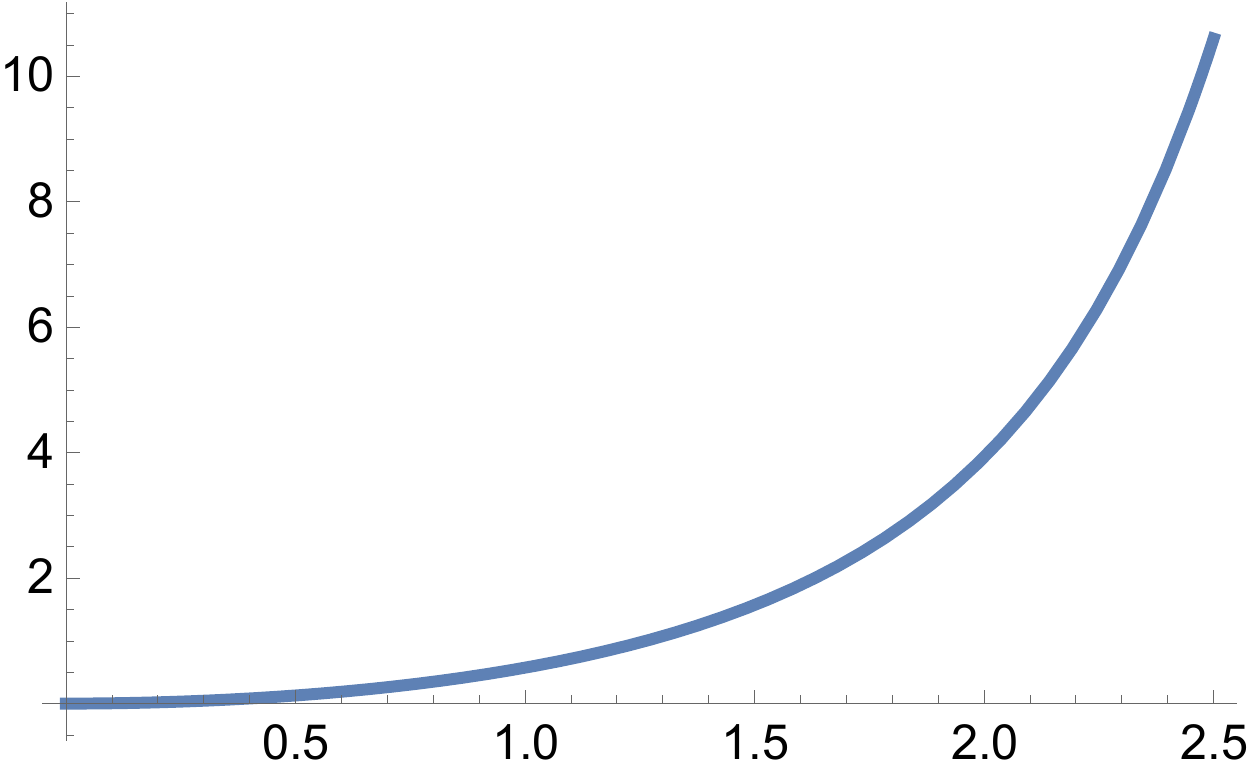}\qquad \includegraphics[width=.34\textwidth]{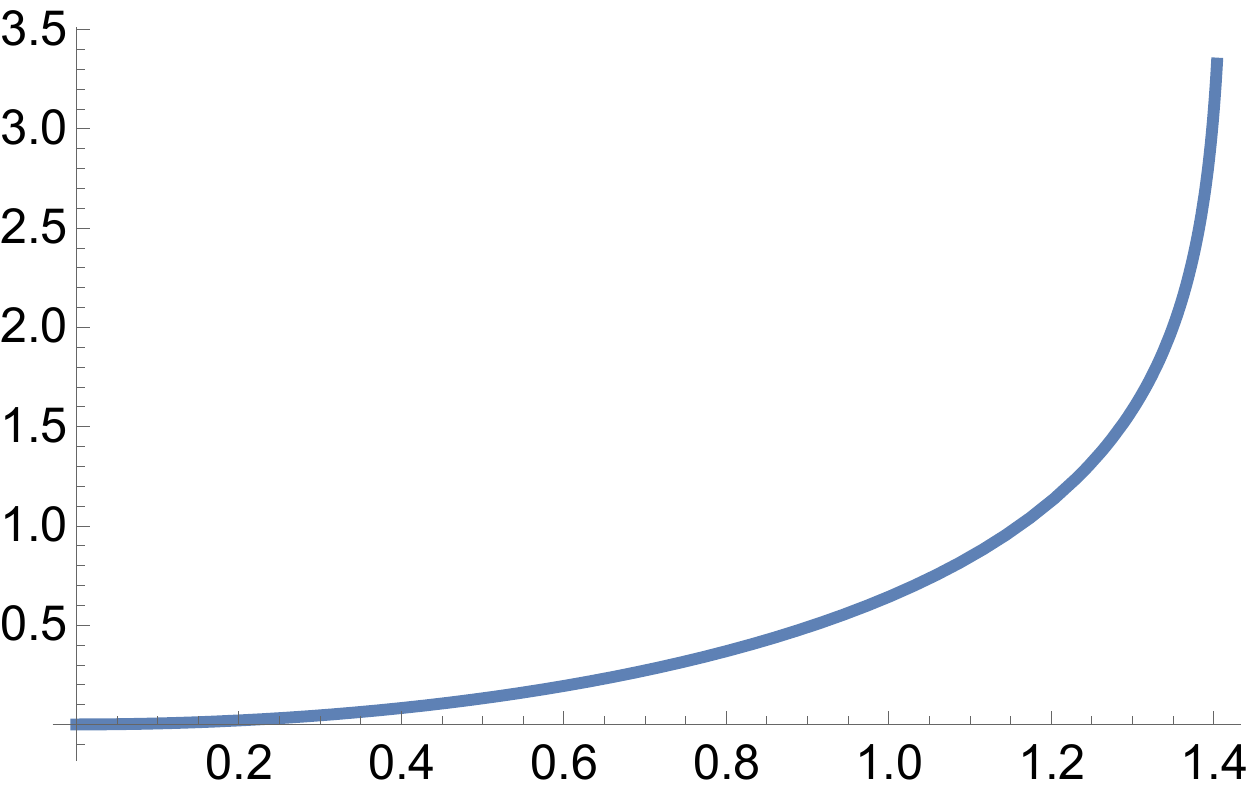}
\end{center}
\caption{Generating curves of rotational $K^\alpha$-translators intersecting orthogonally the rotation axis. Left: $\alpha=1/2$ and the maximal domain is $[0,\infty)$. Right:   $\alpha=1$ and the maximal domain is $[0,\sqrt{2})$.}\label{fig2}
\end{figure}

We point out that in some particular cases, the integrals  in \eqref{p-rot2} can   be explicitly solved.  Here, we denote by $f_\alpha$ to emphasize the parameter $\alpha$ where we also assume  $m=1$. 
\begin{enumerate}
\item Case $\alpha =1$. Then  
$$f^{\prime }(r)=\frac{r\sqrt{4-r^{2}}}{2-r^{2}}$$
and%
\begin{equation*}
f_1(r)=\mp \sqrt{4-r^{2}}\pm \sqrt{2}\tanh ^{-1}\left( \frac{1}{\sqrt{2}}\sqrt{4-r^{2}%
}\right),
\end{equation*}%
defined on $[0,2).$

\item Case $\alpha =1/3$. Now  \begin{equation*}
f^{\prime }(r)=\pm \frac{1}{2}r\sqrt{4+r^{2}}
\end{equation*}%
and the solution is
\begin{equation*}
f_{1/3}(r)=\pm \frac{1}{6}\left( 4+r^{2}\right) ^{3/2},
\end{equation*}%
defined on $[0,\infty ).$

\item   For $\alpha=1/6$, we have
$$f'(r)=\pm (\sqrt{2r^2+1}-1)^{1/2},$$
and
$$f_{1/6}(r)=\pm\frac{\sqrt{\sqrt{2 r^2+1}-1} \left(2 r^2-\sqrt{2 r^2+1}-1\right)}{3 r},$$
and defined on $[0,\infty)$.
\end{enumerate}

\section{Helicoidal $K^\alpha$-translators}\label{sec22}

Consider a helicoidal surface $\Sigma$ in $\r^3$ with axis $z$  whose generating curve $\gamma$ is included in the $xz$-plane and pitch $h$. Without loss of generality, we can assume that $\gamma(r)=(r,0,f(r))$ where $f\colon I\subset\r^+\to\r$ is a smooth function. Then $\Sigma$ parametrizes by 
\begin{equation}\label{p-hel}
X(r,\theta)=(r\cos \theta,r\sin \theta,f(r)+h \theta), \quad r\in I, \theta\in\r.
\end{equation}
If $D=r^2(1+f'^2)+h^2$, then the unit normal vector $N$ is 
$$N=\frac{1}{\sqrt{D}}(h\sin \theta-r f'\cos \theta,-h\cos \theta-rf'\sin \theta,r)$$
and the Gauss curvature is 
$$K=\frac{r^3f'f''-h^2}{D^2}.$$
Then $\Sigma$ is a $K^\alpha$-translator if
\begin{equation}\label{h1}
\left(\frac{r^3f'f''-h^2}{D^2}\right)^\alpha=D^{-1/2}\left(v_1(h\sin \theta-r f'\cos \theta)-v_2(h\cos \theta+rf'\sin \theta)+v_3r\right).
\end{equation}
As a first conclusion, we prove that $\vec{v}$ must be parallel to the $z$-axis. The proof is similar of  Proposition \ref{p-arot}.

\begin{proposition}\label{pr-ahel} Let $\Sigma$ be a $K^\alpha$-translator. If $\Sigma$ is a helicoidal surface,  then  the  axis is parallel to the speed vector $\vec{v}$.
\end{proposition}

\begin{proof} Equation \eqref{h1} can be written as $A_0+A_1 \cos\theta+A_2\sin\theta=0$. From $A_1=A_2=0$, we obtain
$$v_1h-v_2rf'=0,\quad v_1rf'+v_2h=0.$$
Combining both equations, we have $v_2(h^2+r^2f'^2)=0$. Thus $v_2=0$ and hence, $v_1=0$.
\end{proof}

From this proposition, we can assume that $\vec{v}=(0,0,1)$. Then \eqref{h1} is  
\begin{equation}\label{h2}
\left(\frac{r^3f'f''-h^2}{D^2}\right)^\alpha=\frac{r}{\sqrt{D}}.
\end{equation}
We will obtain a first integration of this equation. For this, let 
$$g(r)=r^2(1+f'(r)^2)+h^2.$$
 Then 
$\frac{r}{2}g'=g+r^3f'f''-h^2$ and thus \eqref{h2} is equivalent to
$$rg'=2g+2r^{\frac{1}{\alpha}}g^{\frac{4\alpha-1}{2\alpha}}.$$
If $\alpha=1/2$, the solution is $g(r)=mr^2 e^{r^2}$, with $m\in\r$. If $\alpha\not=1/2$, the solution of this equation is
$$g(r)=r^2\left(\frac{1-2\alpha}{2\alpha}r^2+m\right)^{\frac{2\alpha}{1-2\alpha}},\quad m\in\r.$$
In terms of the function $f$, we have 
$$1+f'^2=\left(\frac{1-2\alpha}{2\alpha}r^2+m\right)^{\frac{2\alpha}{1-2\alpha}}-\frac{h^2}{r^2}.$$
In particular, this gives a restriction on the domain of $f(r)$.

\begin{theorem}\label{t-hel} Let $\Sigma$ be a $K^\alpha$-translator. If $\Sigma$ is a helicoidal surface  about the $z$-axis and pitch $h$, then $\Sigma$   parametrizes as \eqref{p-hel}, where  
\begin{equation}\label{p-hel2}
f(r) =\left\{ 
\begin{array}{lll}
\pm \int^{r}\left(  me^{t^2}-\frac{h^2}{t^2}-1\right) ^{1/2}\, dt, m>0,& &\alpha =\frac{1%
}{2} \\ 
\pm \int^{r}\left( \left(\frac{1-2\alpha}{2\alpha}t^2+m\right)^{\frac{2\alpha}{1-2\alpha}}-\frac{h^2}{t^2}-1\right) ^{1/2}\, dt,m\in \mathbb{R}, & &\alpha \neq \frac{%
1}{2}.%
\end{array}%
\right.
\end{equation}
\end{theorem}

We now show   examples of $f'$ for some choices of $\alpha$. In all cases, the pitch is $h=1$. See Figures \ref{fig3} and \ref{fig4}
\begin{enumerate}
\item Case $\alpha=1/2$. We take $m=1$ in \eqref{p-hel2}. The function $f$ is defined provided $r^2\geq \log((1+r^2)/r^2)$, that is, 
$[r_0,\infty)$, where $r_0\approx 0.898$.   
\item Case $\alpha=1/4$. Let $m=0$. Then $1+f'^2=-1/r^2+r^2$. Now the restriction on $r$ is $r^4-r^2-1\geq 0$, that is, the domain is $[r_0,\infty)$ with $r_0\approx  1.272$. This case appeared in \cite{le}.
\item Case $\alpha=1$. Let $m=0$ in   \eqref{p-hel2}.  Then $1+f'^2=   4/r^4-1/r^2$ and the domain of $f$ is $[0,r_0]$ with $r_0\approx 1.249$. Here $\lim_{r\to 0}f(r)=\lim_{r\to 0}f'(r)=\infty$.
\end{enumerate}

 \begin{figure}[hbtp]
\begin{center}
\includegraphics[width=.3\textwidth]{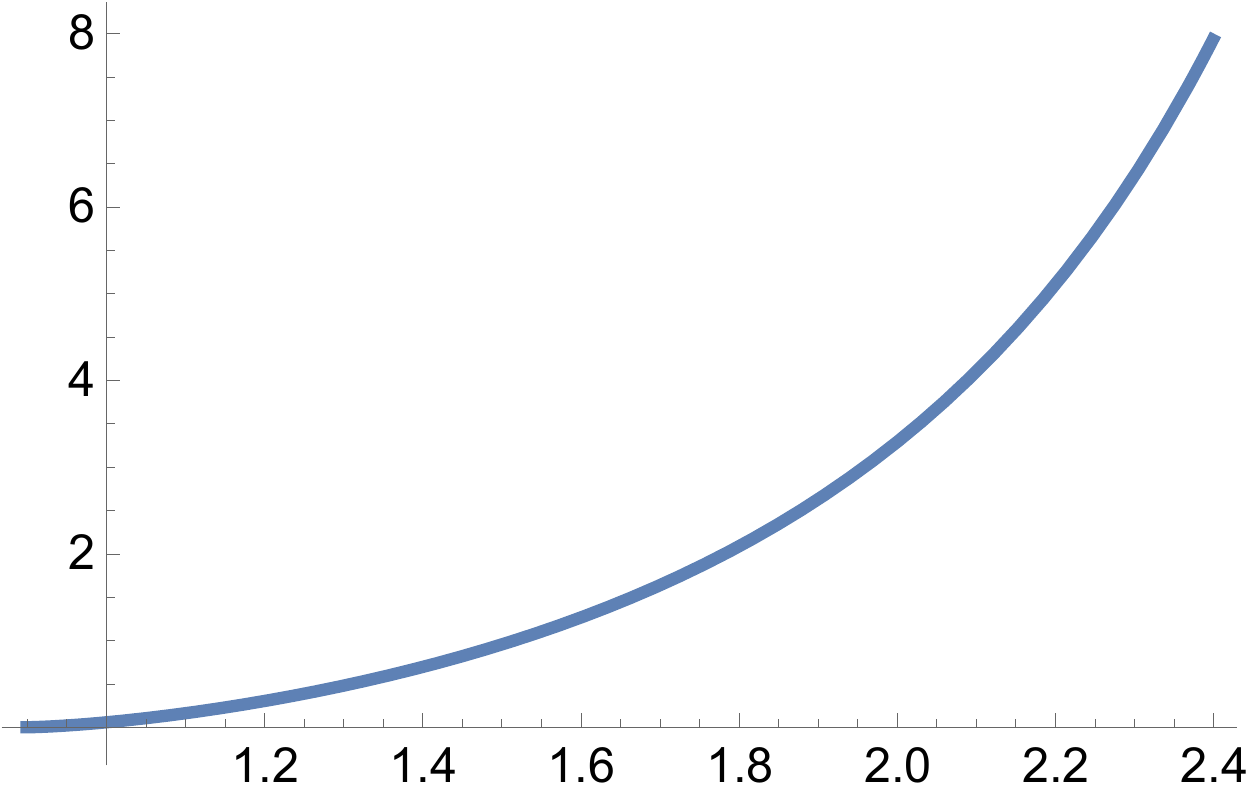}\quad \includegraphics[width=.3\textwidth]{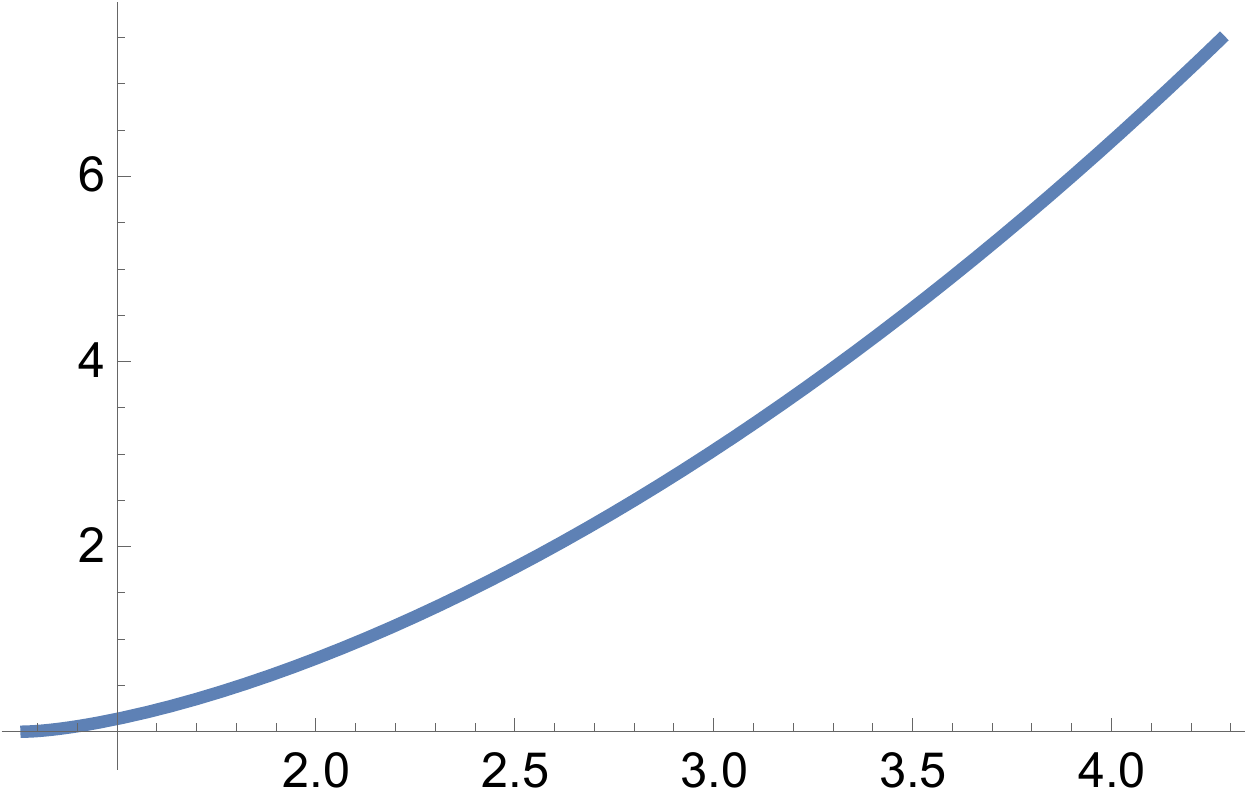}\quad \includegraphics[width=.3\textwidth]{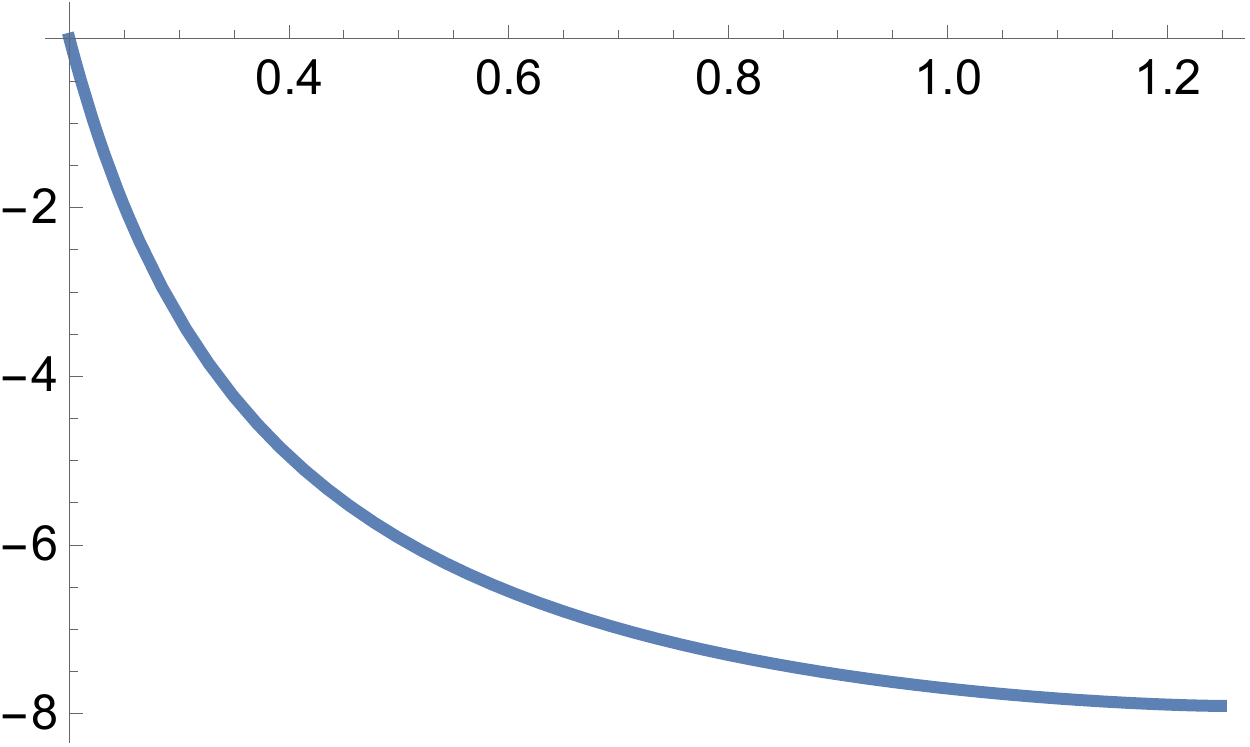}
\end{center}
\caption{Generating curves of helicoidal $K^\alpha$-translators: $\alpha=1/2$ (left), $\alpha=1/4$ (middle) and $\alpha=1$ (right).}\label{fig3}
\end{figure}

 \begin{figure}[hbtp]
\begin{center}
\includegraphics[width=.2\textwidth]{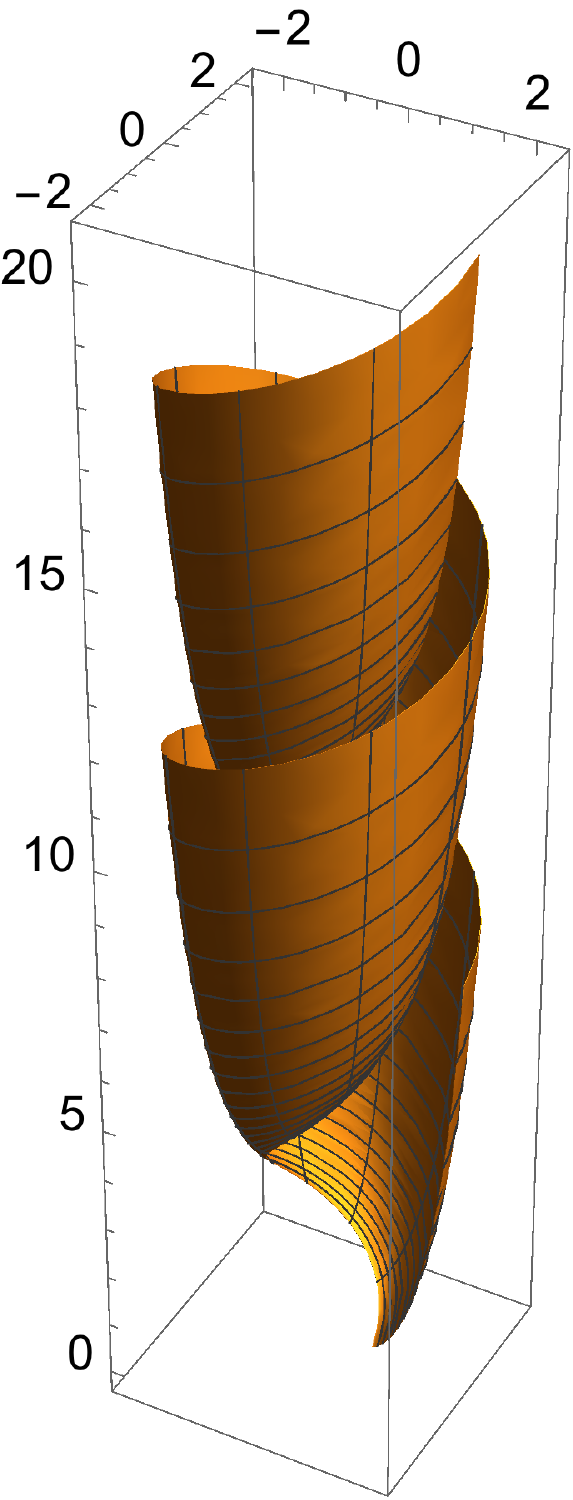}\qquad \includegraphics[width=.27\textwidth]{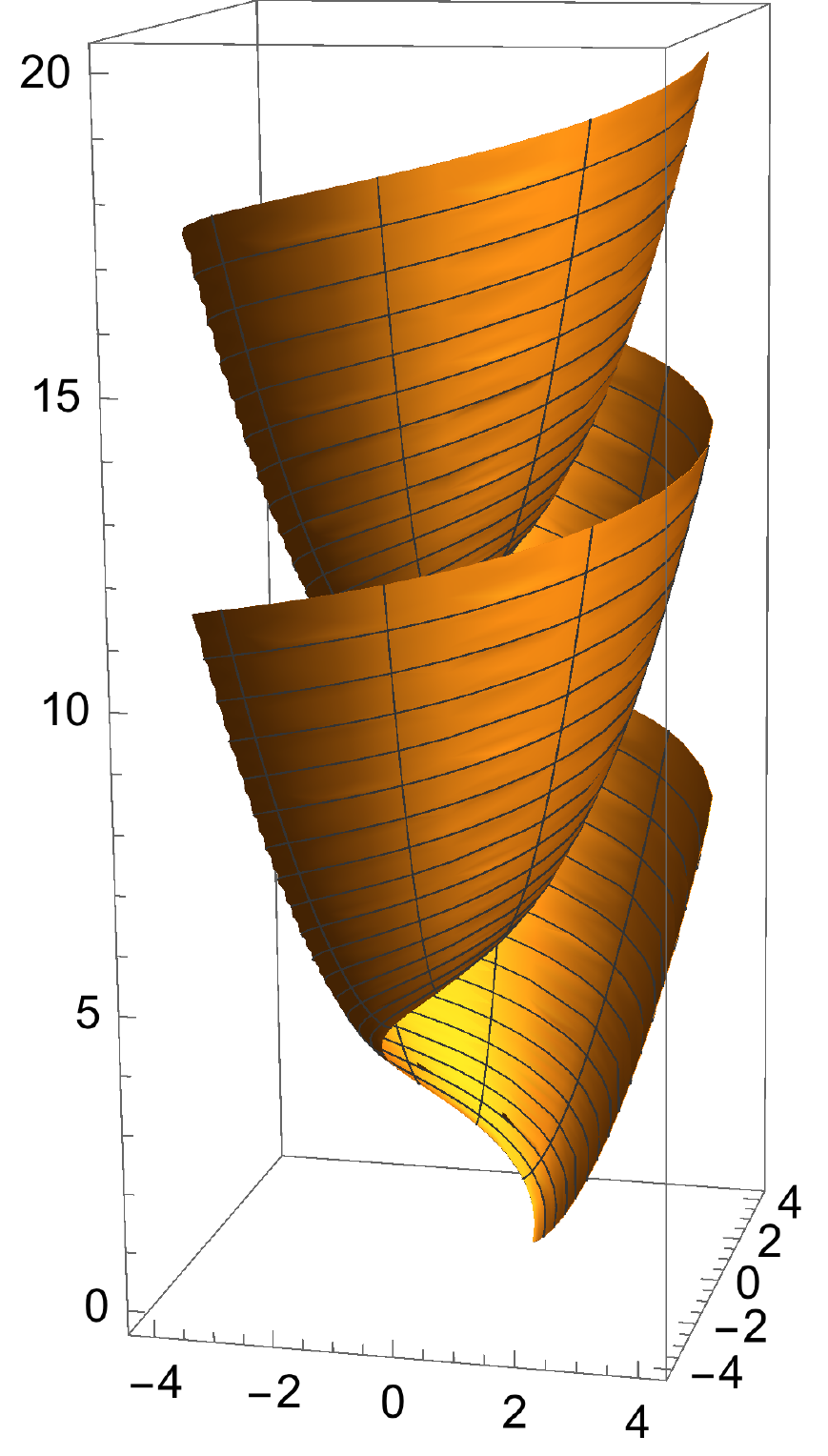}\qquad \includegraphics[width=.12\textwidth]{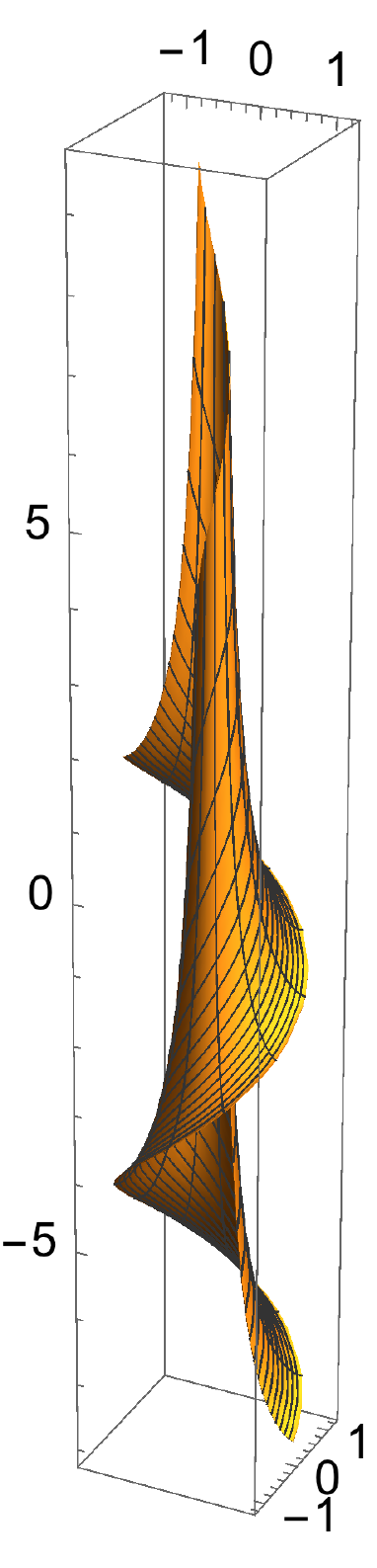}
\end{center}
\caption{Helicoidal $K^\alpha$-translators: $\alpha=1/2$ (left), $\alpha=1/4$ (middle) and $\alpha=1$ (right).}\label{fig4}
\end{figure}

Following Lee \cite{le} we  may  also use the approach of Bour (\cite{dd}).  More explicitly, the\textit{ Bour coordinates} are defined by $r\mapsto s$ and $\theta
\mapsto t,$ $t=\theta +\Theta$, where%
\begin{equation*}
ds=\left( 1+\frac{r^2}{r^2+h^2}f'^2 \right) ^{1/2}dr,\quad
d\Theta =\frac{h}{r^2+h^2}f'dr.
\end{equation*}%
The first fundamental form is now $I=ds^2+U^2 dt^2,$ where the
so-called \textit{Bour function} is introduced by the relation $%
U^2=r^2+h^2.$ Using the Bour function $U$, the terms $r$, $f$
and $\Theta $ can be determined by

\begin{equation}
\left\{ 
\begin{array}{l}
r=\sqrt{U^2-h^2}, \\ 
df^2=\frac{U^2}{(U^2-h^2)^2}\left( U^2\left(
1-(\frac{dU}{ds})^2\right) -h^2\right) ds^2, \\ 
d\Theta =\frac{h}{U^2}df.%
\end{array}%
\right.  \label{h3}
\end{equation}
With the above discussion, the Gauss curvature is now $K=-(d^2 U/ds^2)/U$
and $\left\langle N,(0,0,1)\right\rangle =dU/ds.$ Then \eqref{k1} is now

\begin{equation*}
-\frac{1}{U}\frac{d^2 U}{ds^2}=\left( \frac{dU}{ds}\right) ^{\frac{1}{\alpha }}
\end{equation*}%
or equivalently%
\begin{equation*}
\frac{1}{dU/ds}\frac{d^2 U}{ds^2}=-U\left( \frac{dU}{ds}\right) ^{\frac{%
1-\alpha }{\alpha }}.
\end{equation*}%
Setting $P=dU/ds$ and $\frac{dP}{dU}=(ds/dU)(d^2 U/ds^2)$, we have
\begin{equation*}
\frac{dP}{dU}=-UP^{\frac{1-\alpha }{\alpha }}.
\end{equation*}%
A first integration is%
\begin{equation}
ds=\left\{ 
\begin{array}{lll}
m^{-1}e^{\frac{U^2}{2}}dU,  & &\alpha =1/2 \\ 
(m-\frac{2\alpha -1}{2\alpha }U^2)^{\frac{\alpha }{1-2\alpha }}dU, & &\alpha \neq 1/2,%
\end{array}%
\right.   \label{h4}
\end{equation}%
where $m\in\r$ is a integration constant with $m>0$ if $\alpha=1/2$.
Because $s$ can be viewed a function of $U$, we may interchange their roles. Therefore,  we obtain again a classification of helicoidal $K^{\alpha }$-translators in terms of the Bour function $U$.

\begin{theorem}\label{t-hel2}
Let $\Sigma $ be a $K^{\alpha }$-translator. If $\Sigma $ is a helicoidal
surface about the $z$-axis and pitch $h$, then $\Sigma$    parametrizes as%
\begin{equation*}
X(U,t)=(\sqrt{U^{2}-h^2}\cos (t-\Theta (U)),\sqrt{U^2-h^2}\sin
(t-\Theta (U)),f(U)+h(t-\Theta (U))),
\end{equation*}%
where $U$ is the Bour function, $d\Theta =hU^{-2}df$ and  $df$ is 
\begin{equation}
\left\{ 
\begin{array}{lll}
\frac{\pm m^{-1}U}{U^2-h^2}\left( U^2\left( 1-m^2 e^{-U^2}\right)
-h^2\right) ^{\frac{1}{2}}e^{\frac{U^2}{2}}dU, & &
\alpha =1/2 \\ 
\frac{\pm U}{U^2-h^2}\left( U^2 \left( 1-(m-\frac{2\alpha -1}{2\alpha }%
U^2)^{\frac{2\alpha }{2\alpha -1}}\right) -h^2\right) ^{\frac{1}{2}}(m-%
\frac{2\alpha -1}{2\alpha }U^2)^{\frac{\alpha }{1-2\alpha }}dU, & &\alpha \neq 1/2,%
\end{array}%
\right.   \label{h5}
\end{equation}
with $m\in\r$ ($m>0$ if $\alpha=1/2$).
\end{theorem}

\begin{proof}
Because we see $s$ as a function of $U$ in \eqref{h4}, $U$ can be considered as a new
variable. Then we have the first equality in \eqref{h3},\ where $r$ depends
on this new variable $U$ and as do $f$ and $\Theta .$ Considering this,
together $\theta =t-\Theta $ in \eqref{p-hel}, we have the parametrization of
the helicoidal $K^{\alpha }$-translator. Up to $\alpha =1/2$ or not, from %
\eqref{h4} we complete the proof.
\end{proof}

Again, for some particular values of $\alpha$, \eqref{h4} can be explicitly
integrated. We present the cases $\alpha=1/4$ (\cite{le}), $\alpha=1/3$ and $\alpha=1$.  

\begin{enumerate}
\item Case $\alpha =1/4.$ The solution is 
\begin{equation*}
s=\left\{ 
\begin{array}{lll}
 \frac{1}{2}\left( U\sqrt{m+U^2}+m\cosh ^{-1}(U)\right), & U\geq\{\sqrt{-m},1\} & \mbox{when } m<0, \\ 
 \frac{1}{2}U^2, & U>0 &\mbox{when } m=0, \\ 
\frac{1}{2}\left( U\sqrt{m+U^2}+m\sinh ^{-1}(U)\right), & U>0  &\mbox{when } m>0.%
\end{array}%
\right. 
\end{equation*}

\item Case $\alpha =1/3$. The solution is  
\begin{equation*}
s=mU+\frac{U^3}{6},
\end{equation*}%
where $U\in (0,\infty ).$ 

\item Case $\alpha =1.$ Then,  
\begin{equation*}
s=\left\{ 
\begin{array}{lll}
\sqrt{\frac{2}{m}}\tanh ^{-1}\left( \frac{U}{\sqrt{2m}}\right), & U\leq\sqrt{2m}&\mbox{when } m>0\\
\frac{2}{U}, & U>0&\mbox{when } m=0\\
-\sqrt{\frac{2}{-m}}\tan ^{-1}\left( \frac{U}{\sqrt{-2m}}\right), & U\leq \sqrt{-2m}&\mbox{when } m<0.
\end{array}\right.
\end{equation*}%
In particular, we can express $U$ in terms of $s$,  
\begin{equation*}
U=\left\{ 
\begin{array}{lll}
\sqrt{\frac{m}{2}}\tanh\left(  \sqrt{2m}s\right), &  \mid s\mid\leq (2m)^{-1/2} &\mbox{when } m>0\\
\frac{2}{s}, & s>0&\mbox{when } m=0\\
-\sqrt{\frac{-m}{2}}\tan\left( \sqrt{-2m}s\right), &  \mid s\mid \leq (-2m)^{-1/2} &\mbox{when } m<0.
\end{array}\right.
\end{equation*}%
\end{enumerate}

\section{$K^\alpha$-translators of translation type}\label{sec3}

By a translation surface of $\r^3$ we mean a surface given by the sum of two curves contained in two coordinate planes. This is a particular case of a Darboux surface where $A(t)$ is the identity. After a rigid motion, the surface is the sum of the curves $\gamma(x)= (x,0,f(x))$ and $\beta(y)= (0,y,g(y))$, where $f\colon I\subset\r\to\r$ and $g\colon J\subset\r\to\r$ are smooth functions in one variable. Thus the surface parametrizes  by  
\begin{equation}
X(x,y) =(x,y,f(x) +g(y)),\ x\in I, y\in J. \label{ts1}
\end{equation}%
If  we see the surface $\Sigma$ as the graph of $z=f(x)+g(y)$, then the problem of finding all translation surfaces that are $K^\alpha$-translator is equivalent to ask which are the solutions of \eqref{k1} obtained by separation of variables $z=f(x)+g(y)$. In this section we classify all $K^\alpha$-translators of translation type. 

We calculate all terms of equation \eqref{k1}. Let us observe that once we have the parametrization \eqref{ts1}, we cannot prescribed the speed $\vec{v}$ because the parametrization \eqref{ts1} was previously fixed after a rigid motion.  Thus in order to calculate all terms of \eqref{k1}, the vector $\vec{v}=(v_1,v_2,v_3)$ is assumed in all its generality. The computation of  \eqref{k1} gives  
\begin{equation}
\frac{(f''g'')^\alpha}{(1+f'^2+g'^2)^{2\alpha }}=\frac{v_{3}-v_{1}f'-v_{2}g'}{(1+f'^2+g'^2)^{1/2}},  \label{ts2}
\end{equation}%
where the prime denotes the derivatives with respect to the corresponding variables $x$ or $y$ in each case. 

In order to clarify the arguments, we separate the case $\alpha=1/4$. This case appears because the denominators in \eqref{ts2} are cancelled. 

\begin{theorem}\label{t-31}
 Let $\Sigma$ be a $K^{1/4}$-translator with speed $\vec{v}$ of translation type parametrized by \eqref{ts1}. Up to a change of the roles of $f$ and $g$,  the function $f$ is $f(x)=x^2/m+ax+b$, $a,b,m\in\r$, $m\not=0$ and  $g$ is one of the following functions depending on the speed $\vec{v}$:
\begin{enumerate}
\item If the speed is $\vec{v}=(0,0,1)$, then $g(y)=my^2/4+cy+d$, $c,d\in\r$.
\item If the speed is $\vec{v}=(0,1,v_3)$, then 
 \begin{equation}\label{ts-g}
g(y)=\frac{v_3 (2 c+m y)}{m}-\frac{\left(\frac{3}{2}\right)^{2/3} (2 c+m y)^{2/3}}{m }+d,
\end{equation}
where $c,d\in\r$.

\end{enumerate}
\end{theorem}
 
 \begin{proof}
Now \eqref{ts2} is 
\begin{equation}\label{t-14} 
(f''g'')^{1/4} = v_{3}-v_{1}f'-v_{2}g' .
\end{equation}
Let us observe the symmetry of the roles of $f$ and $g$, so it is enough to  distinguish   cases according to the function $f$. Notice that $f''$ is not constantly $0$ because the case  $K=0$ was initially discarded.
\begin{enumerate}
\item Case $f''=2/m\not =0$ is a non-zero constant, $m\not=0$. In particular, $f(x)=x^2/m+ax+b$, $a,b\in\r$. From \eqref{t-14}, $v_1=0$ and 
$(2g''/m)^{1/4}=v_3-v_2g'$. Since $g''\not=0$ (otherwise, $K=0$),   then $g''=\frac{m}{2}(v_3-v_2 g')^4$.  If $v_2=0$,  we assume that $v_3=1$ and the solution is $g(y)=my^2/4+cy+d$, $c,d\in\r$. If $v_2\not=0$, then the solution of this equation is \eqref{ts-g}.

\item Suppose that $f''$ is not constant. Differentiating successively \eqref{t-14} with respect to $x$ and next with respect to $y$, we obtain   $f'''(x)g'''(y)=0$ for all $x,y$. If at some $x$, $f'''(x)\not=0$, then $g''$ is constant in some interval and we are in the previous case (1) interchanging the roles of $f$ and $g$. Thus $f'''=0$ in its domain, which it is a contradiction because $f''$ is not a constant function.
\end{enumerate}
 \end{proof}

Let us observe that the surfaces of the case (1) of Theorem \ref{t-31} are affinities of the paraboloid $z=x^2+y^2$ and that the speed $\vec{v}=(0,0,1)$. Other consequence   of this result is that  we find $K^{1/4}$-translators where the speed $\vec{v}$ is parallel to the $xy$-plane by choosing $\vec{v}=(0,1,0)$. If we assume that the speed $\vec{v}$ is $(0,0,1)$, as usually is taken as a convention in the literature  (e.g. \cite{cdl,ur2}), then changing the roles of $y$ and $z$, we can provide  examples of translation surfaces $X(x,z)=(x,f(x)+g(z),z)$ whose speed is $(0,0,1)$. 

\begin{example} Let $\alpha=1/4$ and $\vec{v}=(0,0,1)$. Suppose that a $K^{1/4}$-translator  parametrizes as $y=f(x)+g(z)$. In this case, \eqref{k1} is 
$$(f''g'')^{1/4}=g'.$$
Hence $g'>0$ and
$$\frac{g''}{g'^4}=\frac{1}{f''}=m$$
for some constant $m\not=0$. Integrating,
$$f(x)=\frac{1}{2m}x^2+ax+b,\quad g(z)=-\frac{1}{2m}(-3mz+c)^{2/3}+d,$$
with $-3mz+c>0$, $c,d\in\r$. For example, choose $m=1$ and $a=b=c=d=0$. Then the surface is 
$$X(x,z)=(x,\frac{1}{2}x^2-\frac{1}{2}(-3z)^{2/3},z).$$
If we see this surface as the graph on the $xy$-plane,  and after a symmetry about the $z$-plane, we have 
$$z=\frac13(x^2-2y)^{3/2}.$$
\end{example}

  We now consider the general case $\alpha\not=1/4$.
\begin{theorem}\label{t-32}
If $\alpha\not=1/4$, there are not   $K^{\alpha }$-translators that are surfaces of translation. 
  \end{theorem}

\begin{proof} Again by the symmetry of the roles of $f$ and $g$, we discuss according to the function $f$. Recall that $f''$ (and consequently, $g''$) cannot be constantly $0$ because then $K$ would be $0$. Then  \eqref{ts2} is
\begin{equation}
(g'')^\alpha=(1+f'^{2}+g'^2)^{\frac{4\alpha
-1}{2}}(v_{3}-v_{1}f'-v_{2}g')(f'')^{-\alpha }.  \label{ts4}
\end{equation}%

We differentiate \eqref{ts4} with respect to $x$, and after some manipulations, we arrive to 
$$(4\alpha-1)f'f''P-(1+f'^2+g'^2)\left(v_1 (f'')^{1-\alpha}+\alpha (f'')^{-\alpha-1}f'''P\right)=0,$$
where $P=v_3-v_1f'-v_2g'$. The above equation is a polynomial equation on $g'=g'(y)$ of degree $3$, which we write as 
$$A_0+A_1g'+A_2g'^2+A_3g'^3=0,$$
where   all coefficients $A_i$ are functions on the variable $x$. Therefore they must vanish because $g'\not=0$. We have 
\begin{equation*}
\begin{split}
A_1&=-v_2(4\alpha-1)f'f''+\alpha v_2  (1+f'^2)(f'')^{-\alpha-1}f''',\\
A_3&=\alpha (f'')^{-\alpha-1}f'''v_2.
\end{split}
\end{equation*}
 Since $4\alpha-1\not=0$, we obtain $v_2=0$. Now 
\begin{equation*}
\begin{split}
A_0&=(4\alpha-1)(v_3-v_1f')f'f''-(1+f'^2)(v_1 (f'')^{1-\alpha}+\alpha (f'')^{-\alpha-1}f'''(v_3-v_1f')),\\
 A_2&=v_1 (f'')^{1-\alpha}+\alpha (f'')^{-\alpha-1}f'''(v_3-v_1f').
 \end{split}
\end{equation*}
Using $A_2=0$ into $A_1$, we arrive to $(4\alpha-1)(v_3-v_1f')f'f''=0$, obtaining a contradiction. This completes the proof.
\end{proof}

Translations surfaces appear as surfaces obtained by the translation of a curve along another curve. As we have seen, in case that the two curves are included in coordinate planes, then the surface can be written as $z=f(x)+g(y)$. As we said, the problem to classify all translation surfaces that are $K^\alpha$-translators is equivalent to solve equation \eqref{k1} by separation of variables. Other way of separation of variables is assuming that $z=f(x)g(y)$, for two smooth functions $f$ and $g$. However, the equation \eqref{k1} is difficult to solve in all its generality. We only show an example where we can obtain non-trivial examples if  $\alpha=1/4$.

\begin{example}\label{ex-homo}. Assume $\alpha=1/4$ and the speed is $\vec{v}=(0,0,1)$. Instead to assume $z=f(x)g(y)$, we suppose $x=f(z)g(y)$. Then the parametrization of the surface is $X(z,y)=(f(z)g(y),y,z)$, $z\in I$, $y\in J$ and \eqref{k1} is%
\begin{equation}
fgf^{\prime \prime }g^{\prime \prime }-f^{\prime 2}g^{\prime 2}=f^{\prime
4}g^{4}.  \label{h7}
\end{equation}%
Because $\alpha=1/4$, then $K>0$ and so, $f''(x)g''(y)\not=0$. We divide \eqref{h7} with $f^{\prime 2}gg^{\prime \prime },$ obtaining%
\begin{equation}
\frac{ff^{\prime \prime }}{f^{\prime 2}}-\frac{g^{\prime 2}}{gg^{\prime
\prime }}=f^{\prime 2}\frac{g^{3}}{g^{\prime \prime }}.  \label{h8}
\end{equation}%
Differentiating successively \eqref{h8} with respect to $z$ and $y$, we have
\begin{equation*}
2f^{\prime }f^{\prime \prime }\left( \frac{g^{3}}{g^{\prime \prime }}\right)
^{\prime }=0.
\end{equation*}%
Because $f'f''\not=0$, then  $g^{\prime \prime }=ag^{3}$, where  $a\in \mathbb{R},$ $a\neq 0.$ Now %
\eqref{h8} is%
\begin{equation*}
\frac{ff^{\prime \prime }}{f^{\prime 2}}-\frac{f'^2}{a}= \frac{g^{\prime 2}}{ag^{4}}=b, 
\end{equation*}%
for some nonzero constant $b$.  Then $g'^2=abg^4$ and differentiating and using  that $g''=ag^3$, we deduce that $b=1/2$. For the function $g$, we have  
\begin{equation*}
g^{\prime 2}=\frac{a}{2}g^{4},
\end{equation*}%
in particular, $a>0$. The solution of this equation is   $g(y)=\pm \frac{2}{ay+c},$ $c\in \mathbb{R}$. 
We come back to \eqref{h7}, obtaining%
\begin{equation*}
ff^{\prime \prime }-\frac{1}{2}f^{\prime 2}=\frac{f'^4}{a}.
\end{equation*}%
We see $f'$ as a function of $f$, $f^{\prime }=p(f)$. Then $p^{\prime }=\frac{dp}{df}=\frac{f^{\prime
\prime }}{f^{\prime }}$, so
\begin{equation*}
p^{\prime }-\frac{p}{2f}=\frac{p^3}{af},
\end{equation*}%
which is an ODE of Bernoulli type. The solution is%
$$f'(z)=\pm\frac{\sqrt{a}m\sqrt{f}}{\sqrt{1-2m^2f}}.$$
The solution $f(z)$ of this equation together the above function $g(y)$ provide an example of a $K^{1/4}$-translator given by $x=f(z)g(y)$.
\end{example}
\section{Ruled $K^\alpha$-translators}\label{sec4}

In this section we study the solutions of \eqref{k1} when the surface is ruled. A parametrization of a ruled surface $\Sigma$ is 
\begin{equation*}
X(s,t)=\gamma(s)+t w(s),\quad s\in I\subset\r, t\in\r,
\end{equation*}
where $w=w(s)$ is a smooth function with $\vert w(s)\vert=1$ for all $s\in I$ and  $\gamma$
is a curve parametrized by arc-length.

A  case to discard is that $\Sigma$ is a cylindrical surface because   $K=0$. Thus,  $w=w(s)$ is a non-constant function. In such a case, we can choose $\gamma$ to be the   striction curve, that is, $\langle \gamma'(s), w'(s) \rangle =0$ for all $s\in I$. Then 
\begin{equation*}
K=-\frac{\lambda^2}{(\lambda ^2+t^2)^2},\quad \lambda=\frac{(\gamma',w,w')}{\vert w'\vert^2},
\end{equation*}%
 where the parenthesis $(\cdot,\cdot,\cdot)$ denotes the determinant of the three vectors. In particular, we are assuming that $\lambda$ is not constantly $0$. On the other hand, the unit normal vector $N$ is 
\begin{equation*}
N=\frac{\lambda w'+tw'\times w}{\vert w'\vert \sqrt{%
\lambda^2+t^2}}.
\end{equation*}%
In particular, $K$ is negative, so we are assuming that $\alpha$ is an integer.  Then \eqref{k1} writes as 
\begin{equation*}
( -1) ^\alpha\frac{\lambda ^{2\alpha }}{(\lambda^2+t^2)^{2\alpha }}=\frac{1}{%
\vert w'\vert\sqrt{\lambda ^2+t^2}}(\lambda \langle w',\vec{v}%
\rangle +t(w',w,\vec{v})).
\end{equation*}%
or equivalently, 
\begin{equation}\label{eq-ruled}
-(-1)^{\alpha }\lambda ^{2\alpha }\vert w'\vert+\lambda \langle w',%
\vec{v}\rangle (\lambda ^2+t^2)^{\frac{4\alpha -1}{2}}+(w',w,\vec{v}%
)t(\lambda ^2+t^2)^{\frac{4\alpha -1}{2}}=0.
\end{equation}%
 The Wronskian of the functions $\left\{ 1,(\lambda
^2+t^2)^{\frac{4\alpha -1}{2}},t(\lambda ^2+t^2)^{\frac{4\alpha -1}{2}%
}\right\} $ is%
$$(4\alpha -1)(\lambda ^{2}+t^{2})^{4\alpha -3}\left( 4\alpha
t^{2}-\lambda ^{2}\right).$$
Since  $\alpha \neq 1/4$, the Wronskian if not $0$, proving these three functions are   linearly independent. This yields a contradiction with \eqref{eq-ruled}. As a conclusion, we have the following result.

\begin{theorem}\label{t-4} There are not $K^\alpha$-translators that are ruled surfaces.  
\end{theorem}

Let us notice that this result of non-existence has discarded the case that the surface is cylindrical (the rulings are parallel to $\vec{v}$) or that $\lambda=0$, that is, the surface is  developable. The condition $\lambda=0$ together $(w,w',\vec{v})=0$ is equivalent to the case that the base curve is a planar curve parallel to $\vec{v}$, obtaining that the surface is part of a plane.


\section*{Acknowledgments} This publication is part of the Project  I+D+i PID2020-117868GB-I00, supported by MCIN/ AEI/10.13039/501100011033/


\end{document}